\documentclass[a4paper]{article}

\usepackage{amsthm,amsfonts,amssymb,amsmath,amsxtra}
\usepackage[all]{xy}
\usepackage{xr-hyper}
\usepackage[colorlinks=true]{hyperref}
\usepackage{mathrsfs}

\theoremstyle{plain}
\newtheorem{thm}{Theorem}[section]
\newtheorem{lem}[thm]{Lemma}
\newtheorem{cor}[thm]{Corollary}
\newtheorem{prop}[thm]{Proposition}

\theoremstyle{definition}
\newtheorem{defn}[thm]{Definition}

\theoremstyle{remark}
\newtheorem{remark}[thm]{Remark}

    \newcommand{\FH}{{\mathbf{H}}} \newcommand{\FG}{{\mathbf{G}}}
    \newcommand{\FT}{{\mathbf{T}}}
    \newcommand{\FU}{{\mathbf{U}}}
    \newcommand{\FP}{{\mathbf{P}}}
    \renewcommand{\FT}{{\mathbf{T}}}\newcommand{\FM}{{\mathbf{M}}}
    \newcommand{\FA}{{\mathbf{A}}}
    \newcommand{\FZ}{{\mathbf{Z}}}
    
    \newcommand{\FX}{{\mathbf{X}}}

    \newcommand{\sP}{{\mathscr{P}}}
    \newcommand{\sH}{{\mathscr{H}}}
    \newcommand{\sC}{{\mathscr{C}}}

    \newcommand{\BC}{{\mathbb {C}}} 
     
    \newcommand{\BG}{{\mathbb {G}}}

     \newcommand{\BN}{{\mathbb {N}}}
     
     \newcommand{\BR}{{\mathbb {R}}}

     \newcommand{\BZ}{{\mathbb {Z}}}

     \newcommand{\CL}{{\mathcal {L}}}
     \newcommand{\CN}{{\mathcal {N}}}

    \newcommand{\fa}{{\mathfrak{a}}}

    \newcommand{\fo}{{\mathfrak{o}}}

    \newcommand{\sL}{{\mathscr{L}}}

    \renewcommand{\d}{{\mathrm{d}}}

    \newcommand{\Gal}{{\mathrm{Gal}}} 
    \newcommand{\GL}{{\mathrm{GL}}}
    \newcommand{\Hom}{{\mathrm{Hom}}}

    \newcommand{\Ind}{{\mathrm{Ind}}} \newcommand{\ind}{{\mathrm{ind}}}

     \newcommand{\N}{{\mathrm{N}}}

    \newcommand{\ob}{{\mathrm{ob}}}

    \newcommand{\Res}{{\mathrm{Res}}}

    \newcommand{\SO}{{\mathrm{SO}}}\newcommand{\Sp}{{\mathrm{Sp}}}

    \newcommand{\U}{{\mathrm{U}}}
    
    \newcommand{\vol}{{\mathrm{vol}}}

    \newcommand{\pair}[1]{\langle {#1} \rangle}

    \newcommand{\incl}{\hookrightarrow}
    \newcommand{\sk}{\medskip}
    \newcommand{\lra}{\longrightarrow}
    \newcommand{\ra}{\rightarrow} 
    
    \newcommand{\bs}{\backslash}

    \newcommand{\s}{\sk\noindent}

    \newcommand{\abs}[1]{\lvert#1\rvert}

\title{Local periods for discrete series representations}
\author{Chong Zhang}
\begin{document}
\date{}
\maketitle

\begin{abstract}
Let $(G,H)$ be a symmetric pair over a $p$-adic field and $\pi$ a
discrete series representation of $G$. In this paper, for some type
of symmetric pairs $(G,H)$, we show that local periods in
$\Hom_H(\pi,\BC)$ can be constructed by integrating the matrix
coefficients of $\pi$ over $H$.
\end{abstract}

\section{Introduction}

Let $F$ be a non-archimedean local field of characteristic 0. Let
$\FG$ be a connected reductive group over $F$ and $\FH$ a unimodular
spherical subgroup of $\FG$, which means that $\FX=\FH\bs\FG$ is a
spherical variety. Write $G=\FG(F)$ and $H=\FH(F)$.

Let $\pi$ be a unitary irreducible admissible representation of $G$
and $V_\pi$ the space of $\pi$. We say that $\pi$ is
$H$-distinguished if the space $\Hom_H(\pi,\BC)$ is nonzero. We call
elements of $\Hom_H(\pi,\BC)$ local periods. If $\pi$ is
$H$-distinguished, how to explicitly construct nonzero local periods
is an important question. This is part of the local theory of
automorphic periods.

A natural way to construct local periods is to consider the
integration of the matrix coefficients of $\pi$ over $H$. More
precisely, let $\FZ$ be the split component of the center of $\FG$
and $\FZ_\FH=\FZ\cap\FH$. Write $Z_H=\FZ_\FH(F)$. Note that if $\pi$
is $H$-distinguished then the restriction of the central character
of $\pi$ to $Z_H$ is trivial. Fix a $G$-invariant hermitian inner
product $\pair{\cdot,\cdot}$ on $V_\pi$. We formally define a
pairing $\sL$ on $V_{\pi}\times V_{\pi}$ by
\begin{equation}\label{equ. definition of L}
\sL(v,u)=\int_{H/Z_H}\pair{\pi(h)v,u}\ \d h. \end{equation} Note
that if $\sL$ is well defined then it is bi-$H$-invariant and the
map $\sL_u$ given by
\begin{equation}\label{equ. definition of Lu}
\sL_u:v\mapsto\sL(v,u),\quad v\in V_\pi
\end{equation}
belongs to $\Hom_H(\pi,\BC)$. If $\sL$ is well defined, we denote by
$$\sH(\pi)=\{\sL_u\}_{u\in V_{\pi}}$$ the subspace of
$\Hom_H(\pi,\BC)$. Then two natural questions arise:
\begin{enumerate}
\item whether $\sL$ is well defined;
\item if $\sL$ is well defined,  whether we have
$\sH(\pi)=\Hom_H(\pi,\BC)$.
\end{enumerate}

In this paper, we restrict ourselves to the following situations:
\begin{itemize}
\item either $\FX$ is a {\em symmetric space} and $\pi$ is a discrete series
representation;
\item or $\pi$ is supercuspidal and $\FX$ is a symmetric space or
{\em wavefront} spherical variety.
\end{itemize}
We always assume that $\FG$ is split when we require $\FX$ to be a
wavefront spherical variety. We refer to \cite[\S2.1]{sv} for the
definition of wavefront spherical varieties.

When $\pi$ is supercuspidal, $\sL$ is always well defined. The
following notions are natural when we consider discrete series
representations.

\begin{defn}\label{defn. H-integrable}
A discrete series representation $\pi$ of $G$ is called {\em
$H$-integrable} if all its matrix coefficients lie in $L^1(H/Z_H)$,
i.e. $\sL$ is well defined.
\end{defn}

\begin{defn}\label{defn. H-strongly discrete}
A symmetric space $\FX=\FH\bs\FG$ is called {\em strongly discrete}
if all discrete series representations of $G$ are $H$-integrable.
\end{defn}

However, in this paper, we need the following notion of very
strongly discreteness which is stronger than the notion of strongly
discreteness.

\begin{defn}\label{defn. very strong}
A symmetric space is called {\em very strongly discrete} if the
linear form $$\CL:\sC(G/Z)\lra\BC,\quad f\mapsto\int_{H/Z_H}f(h)\ \d
h$$ is well defined and continuous, where $\sC(G/Z)$ is the
Schwartz-Harish-Chandra space of $G/Z$.
\end{defn}

Our main theorems are:

\begin{thm}\label{thm. exhaust}
Suppose that $\FX$ is a very strongly discrete symmetric space and
$\pi$ is a discrete series representation. Then
$\sH(\pi)=\Hom_H(\pi,\BC)$.
\end{thm}

\begin{thm}\label{thm. exhaust supercuspidal}
Suppose that $\FX$ is a symmetric space or a wavefront spherical
variety and $\pi$ is a supercuspidal representation. Then
$\sH(\pi)=\Hom_H(\pi,\BC)$.
\end{thm}

\begin{remark}\label{rem. importance of thm}
Theorem \ref{thm. exhaust} shows that any local period in
$\Hom_H(\pi,\BC)$ is given by integrating some matrix coefficients
over $H$, which is an analog of the global automorphic period. This
theorem can be used to study the factorization of the global periods
into local ones. We refer to \cite[\S17]{sv} for a further
discussion on the link between the global and local theory.
\end{remark}

\begin{remark}\label{rem. exhaust symmetric space}
When $\FX$ is a symmetric space and $\pi$ is of the form
$\ind_J^G\kappa$ for some open compact subgroup $J$ of $G$ and some
irreducible smooth representation $\kappa$ of $J$, results of the
above kind have been obtained by Hakim, Mao and Murnaghan. See
\cite[\S8]{mu} for a survey. Our method is different from theirs.
\end{remark}

\begin{remark}\label{rem. exhaust strongly tempered}
When $\FX$ is a strongly tempered spherical variety (cf.
\cite[\S6.2]{sv} for the definition) and $\pi$ is a tempered
representation, the pairing $\sL$ is well defined by the definition.
In this case, Sakellaridis and Venkatesh \cite[Theorem 6.4.1]{sv}
showed that $\Hom_H(\pi,\BC)$ is nonzero if and only if $\sL$ is
nonzero, by using the Plancherel decomposition of $L^2(X)$ with
$X=H\bs G$. When $(\FG,\FH)=(\SO_n\times\SO_{n+1},\SO_n)$ is in the
setting of Gan-Gross-Prasad conjecture for special orthogonal
groups, Ichino and Ikeda \cite[Proposition 1.1]{ii} showed that
$\FX$ is strongly tempered; Waldspurger \cite[Th\'eor\`em 1]{wa3}
showed that $\Hom_H(\pi,\BC)\leq1$ (basing on the method of
\cite{agrs}) and \cite[Proposition 5.6]{wa2} also proved that
$\Hom_H(\pi,\BC)$ is nonzero if and only if $\sL$ is nonzero by a
different method from that of \cite{sv}. Thus, in this case,
$\sH(\pi)=\Hom_H(\pi,\BC)$. When
$(\FG,\FH)=(\U_n\times\U_{n+1},\U_n)$ is in the setting of
Gan-Gross-Prasad conjecture for unitary groups, analogous result was
proved by Beuzart-Plessis (\cite[Theorem 14.3.1]{bp12} and
\cite[Theorem 8.4.1]{bp15}) for both archimedean and non-archimedean
cases.
\end{remark}

\begin{remark}In general, $\sL$ is not well defined. For example, if
$(\FG,\FH)=(\GL_{2n},\Sp_{2n})$, the matrix coefficients of a
discrete series representation may not belong to $L^1(H/Z_H)$.
\end{remark}

\begin{remark}\label{rem. Gurevich-Offen}
In an earlier version of this paper, we only considered the case
when $\pi$ is supercuspidal. At the same time, we began to study the
case when $\pi$ is discrete series. However, at that time, we could
only show some special cases of symmetric spaces are very strongly
discrete and did not prove Theorem \ref{thm. exhaust}. Recently, we
realize how to prove Theorem \ref{thm. exhaust} when $\FX$ is very
strongly discrete. At the same time, Gurevich-Offen \cite{go} give a
criterion for strongly discreteness (cf. \cite[Theorem 4.4]{go}) and
also a sufficient condition for strongly discreteness (cf.
\cite[Corollary 5.4]{go}). We can show that the condition in
\cite[Corollary 5.4]{go} is sufficient and necessary for very
strongly discreteness.
\end{remark}

As a consequence, under the same assumption in Theorem \ref{thm.
exhaust} or Theorem \ref{thm. exhaust supercuspidal}, we have the
following expression (Corollary \ref{cor. spherical character }) for
the spherical character $\Phi_{\pi,\ell}$ associated to
$\ell\in\Hom_H(\pi,\BC)$. Recall that, for $\ell\in\Hom_H(\pi,\BC)$,
the spherical character $\Phi_{\pi,\ell}$ is defined to be the
distribution on $G$ given by
$$\Phi_{\pi,\ell}(f):=\sum_{v\in\ob(\pi)}\ell(\pi(f)v)\overline{\ell(v)},
\quad f\in C_c^\infty(G),$$ where $\ob(\pi)$ is an orthonormal basis
of $V_\pi$. By Theorem \ref{thm. exhaust} or Theorem \ref{thm.
exhaust supercuspidal}, there exists $v_0\in V_\pi$ such that
$\ell=\sL_{v_0}$. The corollary below is analogous to \cite[Theorem
6.1]{mu} and \cite[Lemma A.3]{iz}. For the proof, see that of
\cite[Lemma A.3]{iz}.

\begin{cor}\label{cor. spherical character }
For all $f\in C_c^\infty(G)$, we have
$$\Phi_{\pi,\ell}(f)=\int_{H/H\cap Z}\int_{H/H\cap Z}\left(\int_G
f(g)\phi(h_2gh_1)\ \d g\right)\ \d h_1\ \d h_2,$$ where
$$\phi(g)=\pair{\pi(g)v_0,v_0},\quad g\in G.$$
\end{cor}

\begin{remark}\label{rem. on cor}
Combined with other inputs, Corollary \ref{cor. spherical character
} can be used to study the supports of spherical characters, which
has potential applications in simple relative trace formula. For
example, see \cite[Theorem A.2]{iz}, \cite[Proposition 4.5]{fmw} and
\cite[Theorem 1.2]{zha}.

\end{remark}

The rest of the paper is organized as follows. In \S2, some basic
notions and properties that will be used in the paper are recalled.
In \S3, we study the very strongly discrete symmetric spaces. Some
aspects of Gurevich-Offen's work will be briefly reviewed there. In
order to prove Theorem \ref{thm. exhaust supercuspidal}, we also
recall a property of supercuspidal representation (see Lemma
\ref{lem. supercuspidal}). In \S4, we prove Theorem \ref{thm.
exhaust}. Our proof is motivated by those of \cite[Proposition
5.6]{wa2} and \cite[Theorem 8.4.1]{bp15}. The proof of Theorem
\ref{thm. exhaust supercuspidal} is the same but much more easier.
We leave it to the reader.

\paragraph{Acknowledgements.} The author would like to thank Maxim Gurevich
and Omer Offen for very helpful discussion on their work \cite{go}
and to Rapha\"{e}l Beuzart-Plessis for kindly sending a preliminary
version of \cite{bp15} before it is available online. He also thanks
Wen-Wei Li and Yiannis Sakellaridis for useful communications.

\section{Notations and preliminaries}
\paragraph{Estimates} Let $S$ be a set. If $f_1$ and $f_2$ are
positive functions on $S$, we write $f_1\prec f_2$ if there exists
$c>0$ such that $f_1(s)\leq cf_2(s)$ for all $s\in S$; we write
$f_1\asymp f_2$ if both $f_1\prec f_2$ and $f_2\prec f_1$.

\paragraph{Fields}
Let $F$ be a non-archimedean local field of characteristic 0. Denote
by $\abs{\cdot}_F$ the normalized absolute value of $F$, and by
$\fo_F$ the ring of integers of $F$.

\paragraph{Groups}
All the algebraic groups mentioned in the paper are defined over
$F$. We use boldface letter to denote an algebraic group, and use
the corresponding nonbold letter to denote the associated group of
$F$-rational points. For an algebraic variety $\FX$ over $F$, $X=
\FX(F)$ is equipped with the natural topology induced from $F$,
which is a locally compact totally disconnected topological space.

For a connected reductive group $\FG$, we denote by $\fa_G$ the real
vector space $\Hom_\BZ(X^*(\FG),\BR)$ where
$X^*(\FG)=\Hom_F(\FG,\BG_m)$ is the group of rational characters of
$\FG$. We have the Harish-Chandra map $H_G:G\ra\fa_G$ defined by
$$\pair{\chi,H_G(g)}=\log(\abs{\chi(g)}_F),\quad g\in G,\ \chi\in
X^*(\FG).$$ Let $\FA_0$ be a maximal split torus of $\FG$ and
$\FP_0=\FM_0\FU_0$ a minimal parabolic subgroup such that
$\FA_0\subset\FM_0$ where $\FU_0$ is the unipotent radical of
$\FP_0$ and $\FM_0$ is the Levi subgroup. There is a canonical
identification $$\fa_0:=\fa_{A_0}\simeq\fa_{M_0}.$$ Denote
$H_0=H_{M_0}$ for short. Let $\Delta(A_0,P_0)$ be the set of simple
roots of $\FA_0$ in the Lie algebra of $\FP_0$. Set
$$M_0^+=\left\{m\in M_0;\ \pair{\alpha,H_0(m)}\geq0,\forall\
\alpha\in\Delta(A_0,P_0)\right\}.$$ Let $K$ be an $A_0$-good maximal
compact subgroup of $G$ and $M_0^1$ the kernel of $H_{M_0}$. Then we
have a Cartan decomposition $$G=\bigsqcup_{m\in M_0^+/M^1_0}KmK.$$
Recall that (cf. \cite[\S I.1, (5)]{wa1})
\begin{equation}\label{equ. volume 1}
\vol(KmK)\asymp\delta_0^{-1}(m)
\end{equation}
as functions on $M_0^+$, where $\delta_0=\delta_{P_0}$ is the
modular character of $P_0$. What we actually use is the following
Cartan decomposition (cf. \cite[Theorem V 3.21]{ren}). Let
$A_0^+=A_0\cap M_0^+$ and $A_0^1=A_0\cap M_0^1$, then there exists a
finite subset $F_0$ of $M_0^+$ such that
\begin{equation}\label{equ. cartan decomposition 2}
G=\bigsqcup_{a\in A_0^+/A^1_0}\bigsqcup_{c\in F_0}KacK.
\end{equation}

\paragraph{Some functions on $G$} Let $\FG$ be a connected reductive
group and keep the notations as before. Fix an algebraic embedding
$\tau:\FG\ra\GL_n$ over $F$. We can and do assume that
$K\subset\GL_n(\fo_F)$. Then a norm function $\|\cdot\|$ on $G$ is
given by
$$\|g\|:=\sup_{i,j}\sup\left(\abs{\tau(g)_{ij}}_F,
\abs{\tau(g^{-1})_{ij}}_F\right).$$ Set $\sigma(g)=\log\|g\|$ to be
the log-norm function on $G$. We refer to \cite[\S I.1]{wa1} for the
properties of $\sigma$. Especially, fixing a $W^G$-invariant norm
$\abs{\cdot}$ on $\fa_0$, where $W^G$ is the Weyl group of $A_0$ in
$G$, we have
$$1+\sigma(m)\asymp 1+\abs{H_0(m)}$$ as functions on $M_0$. Also,
for a compact subset $\omega$ of $G$, we have
$$\sup_{\gamma_1,\gamma_2\in\omega}1+\sigma(\gamma_1g\gamma_2)\asymp
1+\sigma(g)$$ as functions on $G$. For an algebraic (affine) variety
$\FX$ over $F$, there is a general notion of norms on $X$. We refer
to \cite[\S18]{kot} for the precise definitions and some important
properties of norms.

Let $\Xi(g)$ be the Harish-Chandra function of $G$ given by
$\Xi(g)=\pair{\pi(g)v_0,v_0}$ where $v_0$ is the unique
$K$-invariant element of $\Ind_{P_0}^G{\bf1}$ such that $v_0(1)=1$.
Recall that $\Xi$ is positive real-valued, bi-$K$-invariant and
there exists $d\in\BN$ such that
\begin{equation}\label{equ. harish fcn
1}\delta_0^{1/2}(m)\prec\Xi(m) \prec
\delta_0^{1/2}(m)(1+\sigma(m))^d\end{equation} as functions on
$M_0^+$ (cf. \cite[Lemme II.1.1]{wa1}). Also recall that for any
$g_1,g_2\in G$ we have (cf. \cite[Lemme II.1.3]{wa1})
\begin{equation}\label{equ. harish fcn equation}
\int_K\Xi(g_1kg_2)\ \d k=\Xi(g_1)\Xi(g_2).
\end{equation}
For a compact subset $\omega$ of $G$, we have (cf. \cite[Lemma
4.2.3]{sil})
\begin{equation}\label{equ. asymp on compact set}
\sup_{\gamma_1,\gamma_2\in\omega}\Xi(\gamma_1g\gamma_2)\asymp\Xi(g)\end{equation}
as functions on $G$.

We denote by $\sC(G)$ the Schwartz-Harish-Chandra space of $G$,
which is the space of bi-$J$-invariant continuous functions $f$ on
$G$ for some open compact subgroup $J$ of $G$ such that for each
$r\in\BR$
\begin{equation}\label{equ. schwartz-harish-chandra}
\abs{f(g)}\prec\Xi(g)(1+\sigma(g))^{-r}
\end{equation}
as functions on $G$. The space $\sC(G)$ is a locally convex and
compact topological vector space. We refer to \cite[\S II.1]{wa1}
for the precise description of the topology on $\sC(G)$.

\paragraph{Symmetric spaces}
Let $\FG$ be a connected reductive group and $\theta$ a rational
involution of $\FG$ defined over $F$. Let $\FG^\theta$ be the group
of the fixed points of $\theta$ and $(\FG^\theta)^0$ be the
connected component of $\FG^\theta$ containing the identity. Let
$\FH$ be the subgroup of $\FG$ such that
$(\FG^\theta)^0\subset\FH\subset\FG^\theta$. Then $(\FG,\FH)$ is
called a symmetric pair and the geometric quotient $\FX=\FH\bs\FG$
is called a symmetric space.

A split torus $\FA$ of $\FG$ is called $\theta$-split if
$\theta(a)=a^{-1}$ for any $a\in\FA$. A parabolic subgroup $\FP$ of
$\FG$ is called a $\theta$-parabolic subgroup if $\FP$ and
$\theta(\FP)$ are opposite parabolic subgroups. In such a case, we
always take $\FM=\FP\cap\theta(\FP)$ for a Levi subgroup of $\FP$,
which is $\theta$-stable. It is known that $HP$ is open in $G$ when
$\FP$ is a $\theta$-parabolic subgroup. Let $\FP=\FM\FU$ be a
$\theta$-parabolic of $\FG$ and $\FA_{\FP,\theta}$ the maximal
$\theta$-split torus of the center of $\FM$. Denote by
$\Delta(A_{P,\theta},P)$ the set of simple roots of
$\FA_{\FP,\theta}$ in the Lie algebra of $\FP$ and set
$$A^+_{P,\theta}=\left\{a\in A_{P,\theta};\ \abs{\alpha(a)}_F\leq1\
\forall\ \alpha\in\Delta(A_{P,\theta},P)\right\}.$$ We have the
relative Cartan decomposition for symmetric spaces (cf.
\cite[Theorem 1.1]{bo}): there exists a compact subset $\Omega$ of
$G$ and a finite set $\sP$ of minimal $\theta$-parabolic subgroups
of $\FG$ such that
\begin{equation}\label{equ. relative cartan decomposition}
G=\bigcup_{\FP\in\sP}HA^+_{P,\theta}\Omega.
\end{equation}

\paragraph{Some functions on $H\bs G$} Let $(\FG,\FH)$ be a
symmetric pair. We recall some functions on $H\bs G$ introduced in
\cite[\S3]{lag} and their basic properties. Consider the symmetric
map $s:H\bs G\ra G$ given by $s(g)=\theta(g^{-1})g$. The following
functions on $H\bs G$ are all defined by the pullback of some
functions on $G$ via $s$. Set
$$\Theta:=(s^*\Xi)^{\frac{1}{2}},\quad \N_d:=(1+s^*\sigma)^d$$ for
$d\in\BZ$, that is,
$$\Theta(Hg)=\Xi(s(g))^{\frac{1}{2}},\quad
\N_d(Hg)=(1+\sigma(s(g)))^d$$ for $Hg\in H\bs G$. We denote
$\N=\N_1$ for short.

Let $\FP=\FM\FU$ be a minimal $\theta$-parabolic subgroup of $\FG$
and $\omega$ a compact subset of $G$. Choose a norm on $\abs{\cdot}$
on $\fa_M$. Then, as functions on $A_{P,\theta}$, we have (cf.
\cite[Lemma 7]{lag})
\begin{equation}\label{equ. norm functions on symmetric space 1}
\sup_{\gamma\in\omega}\N(Ha\gamma)\asymp1+\abs{H_M(a)}.
\end{equation}
Also, as functions on $A_{P,\theta}$, there exists $d,d'\in\BN$ such
that
\begin{equation}\label{equ. harish-chandra fcn on sym space}
\delta^{\frac{1}{2}}_P(a)\N_{-d}(Ha)\prec\sup_{\gamma\in\omega}
\Theta(Ha\gamma)\prec \delta_P^{\frac{1}{2}}(a)\N_{d'}(Ha)
\end{equation}
(cf. \cite[Proposition 6]{lag}).

We denote by $\sC(H\bs G)$ the Schwartz-Harish-Chandra space of
$H\bs G$ (introduced in \cite[Definition 4.1]{dh}), which is the
space of right-$J$-invariant functions $f$ on $H\bs G$ for some
compact open subgroup $J$ of $G$ such that for any $d\in\BN$
\begin{equation}\label{equ. schwartz-harish-chandra on sym space}
\abs{f(x)}\prec\Theta(x)\N_{-d}(x)\end{equation} as functions on
$H\bs G$.

\paragraph{Representations}
Let $\FG$ be a connected reductive group and $\FZ$ the split
component of the center of $G$. Let $(\pi,V_\pi)$ be an irreducible
admissible representation. The functions $\phi(g)=\pair{\pi(g)v,u}$
on $G$ for some $v,u\in V_\pi$ are called matrix coefficients of
$\pi$. Recall that $\pi$ is called supercuspidal if its matrix
coefficients are compactly supported modulo $Z$, and $\pi$ is called
discrete series if it is unitary and its matrix coefficients are
square-integrable on $G/Z$. It is known that if $\pi$ is discrete
series then the absolute values of all its matrix coefficients
belong to the Schwartz-Harish-Chandra space $\sC(G/Z)$.

Now let $(\FG,\FH)$ be a symmetric pair. For
$\ell\in\Hom_H(\pi,\BC)$ and $v\in V_\pi$, the generalized matrix
coefficient of $\ell$ associated to $v$ is defined by
$$\varphi_{\ell,v}(g)=\ell(\pi(g)v),\quad g\in G,$$ which is a continuous
function on $H\bs G$.

We say that $\pi$ is relatively supercuspidal if its generalized
matrix coefficients $\varphi_{\ell,v}$ are compactly supported
modulo $ZH$ for all $\ell\in\Hom_H(\pi,\BC)$ and $v\in V_\pi$. It is
known that if $\pi$ is supercuspidal then it is relatively
supercuspidal (cf. \cite[Proposition 8.1]{kt1}).

We say that $\pi$ is a relatively discrete series representation if
$\pi$ is unitary and $$\int_{ZH\bs G}\abs{\varphi_{\ell,v}(g)}^2\ \d
g<\infty$$ for all $\ell\in\Hom_H(\pi,\BC)$ and all $v\in V_\pi$. It
is known that if $\pi$ is a discrete series representation then it
is  relatively discrete series (cf. \cite[Proposition 4.10]{kt2}).
It is also known that if $\pi$ is relatively discrete series then
the absolute values of all its generalized matrix coefficients
belong to the Schwartz-Harish-Chandra space $\sC(ZH\bs G)$ (cf.
\cite[Lemma 4.2]{dh}).

\section{$H$-integrability and very strongly discreteness}

\subsection{Supercuspidal representations}
Now let $\pi$ be a supercuspidal representation of $G$.  As
mentioned in the introduction, in such a case, we allow
$\FX=\FH\bs\FG$ to be a wavefront spherical variety but require that
$\FG$ is split. If this is the case, the definition of relatively
supercuspidal is the same as the case of symmetric space. Since
$\pi$ is supercuspidal, the matrix coefficients of $\pi$ are
compactly supported modulo $\FZ$ and thus belong to $L^1(H/Z_H)$.
Therefore the pairing $\sL$ is well defined. To prove Theorem
\ref{thm. exhaust supercuspidal}, we need the following lemma, which
is more or less well known. We present a proof for completeness.

\begin{lem}\label{lem. supercuspidal}
If $\pi$ is supercuspidal, it is relatively supercuspidal.
\end{lem}

\begin{proof}
When $\FX$ is a symmetric space, it is proved by Kato and Takano
\cite[Proposition 8.1]{kt1}.

When $\FG$ is split and $\FX$ is wavefront, the lemma follows from
\cite[Theorem 5.1.2]{sv} on asympotics of the generalized matrix
coefficients. We briefly explain the reason. Given
$\ell\in\Hom_H(\pi,\BC)$ and $v\in V_\pi$, we simply denote
$f=\varphi_{\ell,v}$. Suppose that $v$ is in $V_\pi^J$ where $J$ is
an open compact subgroup of $G$ and $V_\pi^J$ is the subspace of
$V_\pi$ fixed by $J$. For each $\Theta\subset\Delta_X$ where
$\Delta_X$ is the set of spherical roots associated to $\FX$ (see
\cite[\S2.1]{sv}), there is a boundary $\FG$-spherical variety
$\FX_\Theta$ (see \cite[\S2.4]{sv}). Write $X_\Theta=\FX_\Theta(F)$.
The key fact is that, for $\Theta\subsetneq\Delta_X$,
$C^\infty(X_\Theta)$ as a $G$-representation is parabolically
induced from some $P^-_\Theta$-representation where $P^-_\Theta$ is
the parabolic subgroup of $G$ associated to $\Theta$. By
\cite[Theorem 5.1.2]{sv}, for each $\Theta\subset\Delta_X$, there is
a form $\varphi_\Theta\in\Hom_G(\pi,C^\infty(X_\Theta))$ such that
$f|_{N_\Theta}=\varphi_\Theta(v)|_{N'_\Theta}$ where
$N_\Theta,N'_\Theta$ are some ``$J$-good neighborhoods of
$\Theta$-infinity''. Since $\pi$ is supercuspidal, we have
$\varphi_\Theta(v)=0$. Thus, $f|_{N_\Theta}=0$ for each
$\Theta\subsetneq\Delta_X$. At last, Lemma \ref{lem. supercuspidal}
follows from the property that the set
$$X\bs \bigcup_{\Theta\subsetneq\Delta_X}N_\Theta$$ is compact
modulo $Z$.
\end{proof}

\subsection{Very strongly discrete symmetric spaces}
Let $\FH\bs\FG$ be a symmetric space. Since the absolute values of
matrix coefficients of discrete series representations belong to
$\sC(G/Z)$, the symmetric space $\FH\bs\FG$ is strongly discrete if
it is very strongly discrete. As shown in the proof of \cite[Lemma
1]{cl}, $\FH\bs\FG$ is very strongly discrete if and only if the
following condition is
satisfied:\\
($\star$) there exists a natural number $N$ such that
\begin{equation}\label{equ. convergence of xi }
\int_{H/Z_H}\Xi(g)(1+\sigma_*(g))^{-N}\ \d h<+\infty.
\end{equation}
As mentioned in Remark \ref{rem. Gurevich-Offen}, a sufficient
condition for strongly discreteness is obtained by Gurevich-Offen in
\cite[Corollary 5.4]{go}. We will show that this condition is
sufficient and necessary for very strongly discreteness (see
Proposition \ref{prop. criterion}). Before we explain the reason, we
show some examples to illustrate the idea.

\subsubsection{Galois pairs}
Let $\FH$ be a connected reductive group over $F$ and $E$ a
quadratic separable extension of $F$. Let $\FG=\Res_{E/F}(\FH_E)$ be
the Weil restriction of the base change of $\FH$ to $E$, and
$\theta$ the involution on $\FG$ defined by the nontrivial Galois
conjugation in $\Gal(E/F)$. Then $\FH=\FG^\theta$. We say that
$(\FG,\FH)$ is a Galois pair with respect to $E/F$. The global
theory of automorphic period on Galois pairs were studied in
\cite{lr}.

\begin{prop}\label{prop. galois pair}
Let $(\FG,\FH)$ be a Galois pair. Then the symmetric space
$\FH\bs\FG$ is very strongly discrete.
\end{prop}

\begin{proof} The proof is analogous to that of \cite[Proposition
1.1]{ii}. For simplicity, without loss of generality, we assume that
the center of $\FH$ is anisotropic.

Let $\FA_{0,\FH}$ be a maximal split torus of $\FH$ and
$\FP_{0,\FH}$ a minimal parabolic subgroup of $\FH$ containing
$\FA_{0,\FH}$. Then there exists a maximal split torus $\FA_0$ of
$\FG$ which is $\theta$-stable such that $\FA_{0,\FH}\subset\FA_0$
and a $\theta$-stable parabolic subgroup $\FP_1$ of $\FG$ such that
$\FP_{0,\FH}=\FP_1^\theta$. Let $\FP_0$ be a minimal parabolic
subgroup $\FP_0$ such that $\FA_0\subset\FP_0\subset\FP_1$. Let
$A^+_{0,H}$ and $A^+$ be the subsets defined with respect to
$P_{0,H}$ and $P_0$ respectively as in \S2.

As explained in \cite[\S5.1]{go}, we have the relation
\begin{equation}\label{equ. relation for
torus}A^+_{0,H}\subset A^+_0.
\end{equation}
Let $\delta_{0,H},\delta_0$ and $\delta_1$ be the modular characters
of $P_{0,H},P_0$ and $P_1$ respectively. Then we have the relation
(cf. \cite[Lemma 2.5.1]{lr})
\begin{equation}\label{equ. relation for modular characters}
\delta^{\frac{1}{2}}_1|_{P_{0,H}}=\delta_{0,H}\quad
\delta_1|_{A_{0,H}}=\delta_0|_{A_{0,H}}.
\end{equation}

Now let $\Xi$ be the Harish-Chandra function of $G$ defined with
respect to $P_0$ and $A_0$-good maximal compact subgroup $K$. Fix an
$A_{0,H}$-good maximal compact subgroup $K_H$ of $H$. Then by Cartan
decomposition (\ref{equ. cartan decomposition 2}) and relations
(\ref{equ. volume 1}), (\ref{equ. harish fcn 1}) and (\ref{equ.
asymp on compact set}), to show $(\star)$, it suffices to show that
there exists a natural number $N$ such that the following series is
convergent:
\begin{equation}\label{equ. series convergent}\sum_{a\in
A^+_{0,H}/A^1_{0,H}}\delta^{-1}_{0,H}(a)
\delta_0^{\frac{1}{2}}(a)(1+\sigma(a))^{-N}.\end{equation} By the
relation (\ref{equ. relation for modular characters}), such $N$ does
exist.
\end{proof}

\subsubsection{Other cases}
For a general symmetric space $\FH\bs\FG$, it is known that we can
always find such a data: a maximal split torus $\FA_{0,\FH}$ of
$\FH$, a minimal parabolic subgroup $\FP_{0,\FH}\supset\FA_{0,\FH}$,
a $\theta$-stable maximal split torus $\FA_0\supset\FA_{0,\FH}$ of
$\FG$, a $\theta$-stable parabolic subgroup $\FP_1$ of $\FG$ such
that $\FP_{0,\FH}=\FP^\theta_1$, and a minimal parabolic subgroup
$\FP_0$ such that $\FA_0\subset\FP_0\subset\FP_1$.

Suppose that the symmetric space $\FH\bs\FG$ satisfies the following
assumption
\begin{equation}\label{equ. modulus relation 2}
A^+_{0,H}\subset A_0^+,\quad\textrm{and}\quad
\delta_0^{1/2}|_{A^+_{0,H}}\geq\delta_{0,H}|_{A^+_{0,H}}.
\end{equation}
Then, by the same arguments as in the case of Galois pairs, we can
show that $\FH\bs\FG$ is very strongly discrete. For example, the
reader can check that the symmetric pair $(\GL_{2n}(F),\GL_n(E))$
satisfies the assumption (\ref{equ. modulus relation 2}).

\subsubsection{General cases}
In general the condition (\ref{equ. modulus relation 2}) is not
always satisfied. Now we review some part of Gurevich-Offen's work
\cite{go}.

Let $W^H$ be the Weyl group of $A_{0,H}$ in $H$ and $W^{H\bs G}$ the
Weyl group of $A_{0,H}$ in $G$. Then there is a natural embedding
$W^H\subset W^{H\bs G}$. By \cite[Corollary 3.5.(3)]{go}, there is a
particular set of representatives $[W^{H\bs G}/W^H]$ for the coset
$W^{H\bs G}/W^H$. Set $\rho_0^G$ be the usual half sum of positive
roots of $\FA_0$ with respect to $\FP_0$ and $\rho_0^H$ the usual
half sum of positive roots of $\FA_{0,\FH}$ with respect to
$\FP_{0,\FH}$. For $w\in[W^{H\bs G}/W^H]$, set $$\rho_{H\bs
G}^w=\rho_0^G-2w(\rho_0^H),$$ which is viewed as an element in
$\fa_{0,H}^*$ where $\fa_{0,H}^*=X^*(\FA_{0,H})\otimes_\BZ\BR$. Let
$\Delta^{H\bs G}$ be the set of non-zero restrictions to
$\FA_{0,\FH}$ of the simple roots $\Delta(A_0,P_0)$ of $\FA_0$ with
respect to $\FP_0$. We say that $\rho^w_{H\bs G}$ is relativelt
weakly positive if it is a linear combination of the elements of
$\Delta^{H\bs G}$ with non-negative coefficients.

\begin{prop}\label{prop. criterion}
The symmetric space $\FH\bs\FG$ is very strongly discrete if and
only if $\rho^w_{H\bs G}$ is relatively weakly positive for every
$w\in[W^{H\bs G}/W^H]$.
\end{prop}

\begin{proof}
Set $A^+_{e,H}=A^+_{0,H}\cap A^+_0$. Let $\CN^{H\bs G}$ be a subset
of $N_G(A_{0,H})$ consisting of a choice of a representative $n$ for
every element $w\in[W^{H\bs G}/W^H]$. By \cite[Corollary 3.5]{go}
there is a partition for $A^+_{0,H}$:
$$A^+_{0,H}=\bigsqcup_{n\in\CN^{G/H}}n^{-1}A^+_{e,H}n.$$
Then, also by the same arguments as in the case of Galois pairs, we
can show that the condition $(\star)$ holds if and only if there
exists $N\in\BN$ such that the following series is convergent for
each $n\in\CN^{H\bs G}$:
$$\sum_{a\in A^+_{e,H}/(A^+_{e,H}\cap A^1_{0,H})}\delta^{-1}_{0,H}(n^{-1}an)
\delta^{\frac{1}{2}}_0(a)\left(1+\sigma(a)\right)^{-N},$$ which is
equivalent to $$\delta_0^{\frac{1}{2}}(\cdot)|_{A^+_{0,H}}\geq
\delta_{0,H}(n^{-1}\cdot n)|_{A^+_{0,H}}$$ for each $n\in\CN^{H\bs
G}$. The last condition is equivalent to $\rho^w_{H\bs G}$ is
relatively weakly positive for each $w\in[W^{H\bs G}/W^H]$.
\end{proof}
\begin{remark}\label{rem. go's examples}
Gurevich-Offen prove that $\FH\bs\FG$ is strongly discrete if
$\rho^w_{H\bs G}$ is relatively weakly positive for every
$w\in[W^{H\bs G}/W^H]$. They use this condition to show some special
cases of symmetric spaces are strongly discrete. Therefore all the
strongly discrete symmetric spaces in \cite[\S5]{go} are very
strongly discrete series.
\end{remark}

\section{Proof of Theorem \ref{thm. exhaust}}
From now on, let $\FH\bs\FG$ be a very strongly discrete symmetric
space. For simplicity but without loss of generality, we assume that
the center of $\FG$ is anisotropic.

\begin{lem}\label{lem. norm} Let $\FA$ be a
$\theta$-split torus of $\FG$. Then as functions on $H\times A$, we have
$$1+\sigma(ha)\asymp1+\sigma(h)+\sigma(a).$$
\end{lem}
\begin{proof}
Consider the natural map
$$p:\FH\times\FA\lra\FG,\quad (h,a)\mapsto ha.$$
Let $\FT=\FH\cap\FA$, which is a finite group. Then $p$ is a
composition of the quotient map
$$\tau:\FH\times\FA\ra(\FH\times\FA)/\FT$$ and the closed immersion
$$(\FH\times\FA)/\FT\incl\FG,$$ where the action of $\FT$ on
$\FH\times\FA$ is $t\cdot(h,a)=(ht,t^{-1}a)$. Since $\tau$ is a
finite morphism, the map $p$ is also a finite morphism. Note that
$\sigma(h)+\sigma(a)$ is a log-norm on $H\times A$ and $\sigma(ha)$
is the pullback by $p$ of the log-norm $\sigma$ on $G$. Therefore,
by \cite[Proposition 18.1]{kot}, as functions on $H\times A$,
$1+\sigma(ha)$ and $1+\sigma(h)+\sigma(a)$ are equivalent.
\end{proof}

Now let $\FP$ be a minimal $\theta$-parabolic subgroup of $\FG$.
Denote $\FA=\FA_{\FP,\theta}$ and $A^+=A^+_{P,\theta}$.
\begin{lem}\label{lem. asymp of integrals}
Suppose that there exists $N\in\BN$ such that
$$\int_H\Xi(h)(1+\sigma(h))^{-N}\ \d h$$ is convergent.
Then there exists $d\in\BN$ such that, as functions on $A^+$, we
have
$$\int_{H}\Xi(ha)\left(1+\sigma(ha)\right)^{-d}\ \d h\prec\Xi(a).$$
\end{lem}

\begin{proof}
By Lemma \ref{lem. norm}, we have for any $d\in\BN$
$$\left(1+\sigma(ha)\right)^{-d}\prec\left(1+\sigma(h)+\sigma(a)\right)^{-d}
\prec\left(1+\sigma(h)\right)^{-d}.$$ Let $\bar{P}$ be the opposite
of $P$. Let $C_1\subset H,C_2\subset\bar{P}$ be some compact
neighborhoods of the identity. Since $H\bar{P}$ is open in $G$,
there exists a compact neighborhood of the identity $C_K\subset K$
such that $C_K\subset C_1C_2$. Note that there is a compact set $C$
such that $a^{-1}C_2a\subset C$ for any $a\in A^+$. Therefore there
exists $d\in\BN$ big enough such that for any $k=c_1c_2\in C_K$ with
$c_1\in C_1,c_2\in C_2$, we have
$$\begin{aligned}
\int_H\Xi(ha)\left(1+\sigma(ha)\right)^{-d}\  \d h
&\prec
\int_H\Xi(ha)\left(1+\sigma(h)\right)^{-d}\ \d h\\
&\prec
\int_H\Xi(hc_1a\cdot
a^{-1}c_2a)\left(1+\sigma(hc_1)\right)^{-N}\ \d h\\
&=\int_H\Xi(hc_1c_2a)\left(1+\sigma(h)\right)^{-N}\ \d
h.\end{aligned}$$ Therefore, by (\ref{equ. harish fcn equation}), we
have
$$\begin{aligned}
\int_H\Xi(ha)\left(1+\sigma(ha)\right)^{-d}\ \d h &\prec
\int_H\int_{C_K}\Xi(hka)\left(1+\sigma(h)\right)^{-N}\ \d k\ \d h\\
&\prec\int_H\int_{K}\Xi(hka)\left(1+\sigma(h)\right)^{-N}\ \d k\ \d h\\
&=\Xi(a)\cdot\left(\int_H\Xi(h)\left(1+\sigma(h)\right)^{-N}\ \d
h\right),
\end{aligned}$$
which completes the proof.
\end{proof}

\begin{lem}\label{lem. main lemma}
Let $\pi$ be a discrete series representation of $G$. Then
for any generalized matrix coefficient $\varphi$ and matrix
coefficient $\phi$ of $\pi$, the integral $$\int_G\varphi(g)\phi(g)\ \d g$$
is absolutely convergent.
\end{lem}
\begin{proof}
We have to show that the integral $$\int_{H\bs
G}\abs{\varphi(g)}\int_H\abs{\phi(hg)}\ \d h\ \d g$$ is absolutely
convergent. Consider the relative Cartan decomposition (\ref{equ.
relative cartan decomposition}). Let $\FP\in\sP$ and
$\FA=\FA_{\FP,\theta}$. Choose a maximal split torus $\FA_0$ of
$\FG$ and a minimal parabolic subgroup $\FP_0$ of $\FG$ such that
$\FA\subset\FA_0\subset\FP_0\subset\FP$. It suffices to show that
$$\int_{H\bs HA^+\Omega}\abs{\varphi(a)}\int_H\abs{\phi(ha)}\ \d h\ \d a\ \d
k$$ is absolutely convergent. Since $\pi$ is discrete series, it is
also relatively discrete series. Therefore the generalized matrix
coefficient $\varphi$ belongs to $\sC(H\bs G)$. Thus, for any
$d\in\BN$, we have
$$\abs{\varphi(a)}\prec\Theta(a)\N_{-d}(a).$$

Since $\FH\bs\FG$ is very strongly discrete and $\pi$ is discrete
series, by Lemma \ref{lem. asymp of integrals}, for the inner
integral we have
$$\int_H\abs{\phi(ha)}\ \d h\prec\Xi(a).$$
Hence it suffices to show that, for some $d\in\BN$ big enough, the integral
$$\int_{H\bs HA^+\Omega}\Theta(a)\Xi(a)\N_{-d}(a)\ \d a\ \d\gamma$$ is convergent.
Note that $A^+\subset A_0^+$. According to (\ref{equ. harish fcn 1})
and (\ref{equ. harish-chandra fcn on sym space}), there exists
$d_1,d_2\in\BN$ such that
$$\Theta(a)\prec\delta_P^{\frac{1}{2}}(a)\N_{d_1}(a),\quad
\Xi(a)\prec\delta^{\frac{1}{2}}_{P_0}(a)(1+\sigma(a))^{d_2}
=\delta^{\frac{1}{2}}_{P}(a)(1+\sigma(a))^{d_2}.$$ The rest of the
proof is the same as that of \cite[Lemma 2.1]{dh}.
\end{proof}

Let $\pi$ be a discrete series representation of $G$. Now we define
an hermitian inner product $(\cdot,\cdot)$ on $\Hom_H(\pi,\BC)$.
Denote by $\bar{\pi}$ the complex conjugate of $\pi$. For
$\ell_1,\ell_2\in\Hom_H(\pi,\BC)$ and $v,v'\in V_\pi$, the following
integral is well defined:
$$\pair{v,v'}_{\ell_1,\ell_2}:=\int_{H\bs G}\varphi_{\ell_1,v}(g)
\overline{\varphi_{\ell_2,v'}(g)}\ \d g,$$ since the generalized
matrix coefficients are square-integrable over $H\bs G$. Then
$$v\otimes\bar{v'}\mapsto\pair{v,v'}_{\ell_1,\ell_2}$$ defines a
morphism in $\Hom_G(\pi\otimes\bar{\pi},\BC)$. Since
$\dim\Hom_G(\pi\otimes\bar{\pi},\BC)=1$, there exists
$d_{\ell_1,\ell_2}\in\BC$ such that
$$\pair{v,v'}_{\ell_1,\ell_2}=d_{\ell_1,\ell_2}\pair{v,v'}$$ for any
$v,v'\in V_\pi$. Define
\begin{equation}\label{equ. inner product}
(\ell_1,\ell_2)=d_{\ell_1,\ell_2}.\end{equation} It is obvious that
$(\cdot,\cdot)$ is an hermitian inner product on $\Hom_H(\pi,\BC)$.

\begin{proof}[Proof of Theorem \ref{thm. exhaust}]
Let $\sH(\pi)^\bot$ be the orthogonal complement of $\sH(\pi)$ in
$\Hom_H(\pi,\BC)$ with respect to the inner product $(\cdot,\cdot)$
defined as (\ref{equ. inner product}). By \cite[Theorem 4.5]{de} we
have
$$\dim\Hom_H(\pi,\BC)<\infty$$ for any irreducible admissible
representation $\pi$. Hence, to show $\sH(\pi)=\Hom_H(\pi,\BC)$, it
suffices to show that $\sH(\pi)^\bot$ is zero.

Suppose that $\sH(\pi)^\bot$ is nonzero and choose a nonzero element
$\ell$ of $\sH(\pi)^\bot$. We will show that there exists a vector
$v_0$ in $V_\pi$ so that
$$(\ell,\sL_{v_0})\neq0,$$ which is a contradiction.
For $u,v\in V_\pi$, set $\phi(g)=\pair{\pi(g)v,u}$ for the matrix
coefficient associated to $u,v$. For $v_0\in V_\pi$, set
$$I(v_0,\phi)=\int_G\phi(g)\overline{\varphi_{\ell,v_0}(g)}\ \d g,$$
which is absolutely convergent by Lemma \ref{lem. main lemma}.
Therefore
$$I(v_0,\phi)=\int_{H\bs G}\varphi_{\sL_u,v}(g)
\overline{\varphi_{\ell,v_0}(g)}\ \d g=\pair{v,v_0}_{\sL_u,\ell}.$$

On the other hand, $I(v_0,\phi)$ can be rewritten as
$$\lim_{n\ra\infty}\int_{\Omega_n}\phi(g)\overline{\ell(\pi(g)v_0)}\ \d
g,$$ where $\Omega_n$ is an increasing family of open compact
subsets of $G$ and whose union is $G$. The integration over
$\Omega_n$ is actually a finite sum. Thus $\ell$ can be moved out
and we get
$$\begin{aligned}
\int_{\Omega_n}\phi(g)\overline{\ell(\pi(g)v_0)}\ \d g
&=\overline{\ell\left(\int_{\Omega_n}\bar{\phi}(g)\pi(g)v_0\ \d g\right)}\\
&=\overline{\ell(\pi(\bar{\phi}\cdot{\bf1}_{\Omega_n})v_0)},
\end{aligned}$$
where ${\bf1}_{\Omega_n}$ is the characteristic function of
$\Omega_n$. Therefore, passage to $n\ra\infty$, we obtain
$$\begin{aligned}
&\lim_{n\ra\infty}\int_{\Omega_n}\phi(g)\overline{\ell(\pi(g)v_0)}\
\d g\\
=&\lim_{n\ra\infty}\overline{\ell(\pi(\bar{\phi}\cdot{\bf1}_{\Omega_n})v_0)}\\
=&\overline{\ell(\pi(\bar{\phi})v_0)}.
\end{aligned}$$ For the definition of $\pi(\phi)v_0$, we refer to
\cite[\S III.7]{wa1}.

Now we choose some specific $v_0\in V_\pi$ so that
$\ell(v_0)\neq0$ and set $\phi_0(g)=\pair{v_0,\pi(g)v_0}$. Then, by
Schur orthogonality relations, $\pi(\phi_0)v_0=\lambda v_0$ for some
nonzero $\lambda\in\BC$. Thus
$$\pair{v_0,v_0}_{\sL_{v_0},\ell}
=\overline{\ell\left(\pi(\phi_0)v_0\right)}=\overline{\lambda\ell(v_0)}\neq0.$$
Therefore, $(\sL_{v_0},\ell)\neq0$, which completes the proof.
\end{proof}

\s{\small Chong Zhang\\
School of Mathematical Sciences, Beijing Normal University,\\
Beijing 100875, P. R. China.\\
{\em E-mail address}: \texttt{zhangchong@bnu.edu.cn}}

\end{document}